\documentclass[11pt,a4paper,reqno,twoside]{amsart}

\usepackage{epsfig}
\usepackage{xspace}
\usepackage{natbib}
\usepackage{epsfig}
\usepackage{amssymb}
\usepackage{amsthm}
 \newtheorem{theorem}{Theorem}[section]
       \newtheorem{lemma}[theorem]{Lemma}
       \newtheorem{proposition}[theorem]{Proposition}
       \newtheorem{corollary}[theorem]{Corollary}
    \newtheorem{definition}[theorem]{Definition}
       
       \newtheorem{claim}[theorem]{Claim}       
\begin{document}
%\baselineskip 20pt
%\title[Grobner Bases and Maximum Entropy]
%{On Gr{\"{o}}bner bases of Elimination and Toric Ideals describing Maximum
%  Entropy models}
\title[Tropical Algebraic approach to Consensus over Networks]
{Tropical Algebraic approach to \\ Consensus over Networks}
\author[J. G. Manathara]{Joel George Manathara}
%\thanks{The second author is supported in part by the
%       National Science Foundation, grant DMS--9706793. He dedicates this
%	paper to Prof. John Benedetto, on his 60's birthday.} 
\address{Department of Aerospace Engineering \\ Indian Institute of
  Science \\ Bangalore 560012, India.}
\email{joel@aero.iisc.ernet.in}
\author[A. Dukkipati]{Ambedkar Dukkipati}
%\thanks{The material in this paper was presented in part at the IEEE
%  International Symposium on Information
%  Theory (ISIT'2010)~\cite{Dukkipati:2009:EmbeddingMaximumEntropyModels} and in
%  part at the 12th International Workshop on Computer Algebra in
%  Scientific Computing
%  (CASC'2010)~\cite{DukkipatiManathara:2010:AnAlgebraicImplicitizationAndSpecialization}.} 
\address{Department of Computer Science and Automation \\ Indian Institute of
Science \\ Bangalore 560012, India.}
\email{ambedkar@csa.iisc.ernet.in}
\author[D. Ghose]{Dabasish Ghose}
%\thanks{The second author is supported in part by the
%       National Science Foundation, grant DMS--9706793. He dedicates this
%	paper to Prof. John Benedetto, on his 60's birthday.} 
\address{Department of Aerospace Engineering \\ Indian Institute of
  Science \\ Bangalore 560012, India.}
\email{dghose@aero.iisc.ernet.in}
%\date{April 23, 2000}


%======================Abstract=================
\begin{abstract}
In this paper we study the convergence of the max-consensus protocol.
Tropical algebra is used to formulate the problem. Necessary and
sufficient conditions for convergence of the max-consensus protocol
over fixed as well as switching topology networks are given.
\end{abstract}
\maketitle
%\vspace{-2ex}
%================================================
\section{Introduction}
The consensus problem addresses the question of whether it is
possible to achieve a consensus among agents over a network while
communicating only with the immediate neighbors. The consensus
problem in multi-agent or distributed systems were addressed as
early as in mid eighties \citep{tsitsiklis}. The field has been been
dormant for a while, with sporadic appearances of seemingly
unrelated publications \citep{reynolds, vicsek}, until it saw a burst
in the last few years \citep{ali, olfati2,olfati1,ren}. The reason
for a renewed interest in the field of consensus is due to the
applications it finds in decentralized multi-agent coordination
missions and the recent developments of concrete instances of such
systems like wireless sensor networks and multiple unmanned aerial
vehicles.

The average or the weighted average consensus of agents over
networks, where the values of all nodes or agents in the network
converges to a weighted average of their initial values, is what
received wide attention of researches, and this and various
extensions of it have been addressed extensively \citep{olfati1,ren}.
On the otherhand In a max-consensus, every agent in the network,
communicating with 
only the immediate neighbors, should converge to a value which is
maximum of all the initial values of the agents. Max-consensus has
applications like decentralized leader selection and detection of
faults in large networks. The problem of max-consensus over a
network of fixed topology was addressed in \citep{olfati2}. However,
as the max operation is not linear, its treatment in~\citep{olfati2}
is different from that of linear protocols where convergence to
consensus is proved through convergence of matrix products. We
propose a different analysis of max-consensus protocol by forcing
max operation to be linear through the use of tropical algebra
\citep{maxplus}. By this we can give an analysis of the max-consensus
protocol akin to the standard method of convergence
analysis for linear consensus protocols, that is, analysis of
products of adjacency matrices \citep{ali}. We also give
necessary and sufficient conditions for convergence of max-consensus
protocol under switching topology.

The rest of the paper is organized as follows.
First, we state the max-consensus problem in its standard form.
Then, we introduce the max-plus or the tropical algebra and
thereafter reformulate the max-consensus problem. Under this new
setting, we analyze the convergence of max-consensus protocol over a
fixed network and give necessary and sufficient condition for this
to occur. Then we study max-consensus under switching topologies.
The necessary and sufficient condition for convergence is given for
this case also.

%=====================Section: Background and Preliminaries=========
\section{Background and Preliminaries}
\label{Consensus:SubSection:MaxConsensus:ProblemStatement} Let
$\mathcal{G}=(\mathcal{V},\mathcal{E})$ be a directed graph
(digraph) with nodes $\mathcal{V}=\{1,\ldots,N\}$ and edges
$\mathcal{E}\subset \mathcal{V}\times\mathcal{V}$. For a node
$i\in\mathcal{V}$, let $\mathcal{N}_i$ denote the neighbor set of
$i$. We have
\begin{displaymath}
\mathcal{N}_i = \{j: j\in\mathcal{V}, (j,i)\in \mathcal{E}\}.
\end{displaymath}
By $(j,i)\in\mathcal{E}$, we mean an edge directed from node $j$ to
node $i$ (node $i$ receives information from node $j$). We consider
graphs with self loops, that is, if $i\in\mathcal{V}$, then
$(i,i)\in\mathcal{E}$. Let $x_i(k)\in\mathbb{R}$ be the value of
node $i$ at time step $k$. Then, the max-consensus protocol that we
consider, for node $i$, is as follows:
\begin{equation}
\label{Consensus:Equation:MaxConsensus:MaxConsensusProtocol}
  x_i(k+1) = \max_{j\in\mathcal{N}_i}\{x_j(k)\}\enspace, \quad k\in\mathbb{Z}_{\ge 0}.
\end{equation}
Now the max-consensus problem can be stated as follows:
\begin{quote}
{\it Given a digraph $\mathcal{G}=(\mathcal{V},\mathcal{E})$, does
there exist a $n\in\mathbb{N}$ such that for all $k\geq n,$ $
x_i(k)$ $=$ $\max_{j\in\mathcal{V}}\{x_j(0)\}$ for all
$i\in\mathcal{V}$?}
\end{quote}
That is, the value at each node converges to a value which is equal
to the maximum of all the initial nodal values.

%\section{Max-Plus Algebra}
%\label{Consensus:SubSection:MaxConsensus:MaxPlusAlgebra}
The max-consensus problem is nonlinear under usual matrix algebra.
Hence, we switch to an algebra to `linearize' the max-consensus
problem. This will enable us for natural extensions and
generalizations like analysis of the max-consensus under switching
topologies where the underlying graph changes with time. The
suitable framework for the max-consensus problem is the
tropical algebra. Here we give a brief overview of the tropical
algebra (for details, see \citep{maxplus,Izhakian:2008:TropicalAlgebraicSets}).

 In tropical algebra, we consider a
semiring $(\mathbb{R}\cup \{-\infty\},\oplus,\otimes)$ called the
tropical semiring with addition and multiplication defined as
follows. For $a,b \in \mathbb{R}\cup \{-\infty\},$ define $a\oplus b
= \max(a,b)$ and $a\otimes b = a + b$ and one can show that
$(\mathbb{R}\cup \{-\infty\},\oplus,\otimes)$ has a semiring
structure. $(\mathbb{R}\cup \{-\infty\},\oplus,\otimes)$ is called
tropical semiring.
%Note  
%that the additive group $(\mathbb{R}\cup \{-\infty\},\oplus)$ does
%not have inverse elements.
The zero element of this semiring is
$-\infty$ and the multiplicative identity is $0$.

For the purpose of analyzing the max-consensus problem, we require
only a tropical sub-semiring with two elements $0$ and $-\infty$. We
denote this binary tropical sub-semiring as $
(\mathbb{T},\oplus,\otimes)$, where $\mathbb{T}=\{0,-\infty\}$.

%=============================Section: Problem reformulation============
\section{Problem Reformulation}
\label{Consensus:SubSection:MaxConsensus:ProblemReformulation} We
reformulate the max consensus problem using tropical
algebra. Towards this, we define the tropical adjacency matrix of a
graph as follows.
\begin{definition}
{\label{Consensus:Definition:MaxConsensus:TropicalAdjecencyMatrix}
The tropical adjacency matrix of a graph
$\mathcal{G}=(\mathcal{V},\mathcal{E})$ is a matrix $A \in
\mathbb{T}^{N\times N}$, where $\mathcal{V}=\{1,\ldots,N\}$, with
$(i,j)^\mathrm{th}$ entry defined as follows.
\[(A)_{i,j} = \left\{
\begin{array}{ll} 0 & \mbox{if $i=j$, or $(j,i) \in \mathcal{E}$} \\-\infty &
\mbox{otherwise.}
\end{array}
\right.\]}
\end{definition}
Since we consider graphs with self loops, all tropical adjacency
matrices will have diagonal entries as zeros.

The tropical linear algebra operations involving matrices and
vectors is as follows. For tropical adjacency matrices
$A,B\in\mathbb{T}^{N\times N}$ and a vector $\mathbf{x} =
(x_1,\ldots,x_N)\in \mathbb{R}^N$, we have
\begin{align}
(A\oplus B)_{ij} &= a_{ij} \oplus b_{ij} \notag\\
& = \max\{a_{ij},b_{ij}\}
\label{Consensus:Equation:MaxConsensus:TropicalAddition}
\end{align}
\begin{align}
(A\otimes B)_{ij} &= \bigoplus_l \left( a_{il} \otimes b_{lj} \right) \notag\\
& = \max_l\{a_{il}+b_{lj}\}
\label{Consensus:Equation:MaxConsensus:TropicalMultiplication}
\end{align}
\begin{align}
(A\otimes \mathbf{x})_{i} &= \bigoplus_j\left(a_{ij} \otimes {x}_{j}\right) \notag \\
& = \max_j\{a_{ij}+{x}_{j}\},
\label{Consensus:Equation:MaxConsensus:TropicalVectorMultiplication}
\end{align}
where $a_{ij}=(A)_{ij}$ and $b_{ij}=(B)_{ij}$ with $(\cdot)_{ij}$
denoting the $(i,j)^{\rm th}$ element. The max-consensus protocol in
Eq.~\eqref{Consensus:Equation:MaxConsensus:MaxConsensusProtocol} can
now be rewritten using tropical algebraic operations as follows.
\begin{align}
x_i(k+1) &= \max_{j\in\mathcal{N}_i} \left\{ x_j(k)\right\} \notag\\
& = \max \left\{
\max_{j\in\mathcal{N}_i}\{0+x_j(k)\},\max_{j\in\mathcal{V}\setminus
\mathcal{N}_i}\{-\infty+x_j(k)\}\right\}\enspace.
\label{Consensus:Equation:MaxConsensus:MaxExpanded}
\end{align}
If $A$ is the tropical adjacency matrix corresponding to graph
$\mathcal{G}$, then from
Definition~\ref{Consensus:Definition:MaxConsensus:TropicalAdjecencyMatrix}
and Eq.~\eqref{Consensus:Equation:MaxConsensus:MaxExpanded}, it
follows that
\begin{align*}
x_i(k+1) &= \max_{j\in\mathcal{V}} \left\{(A)_{i,j} + x_j(k)\right\}
\\
& = \bigoplus_{j=1}^n \left((A)_{i,j}\otimes x_{j}(k)\right).
\end{align*}
Using the notation $\mathbf{x}(k) = (x_1(k),\ldots,x_N(k))^T$, we
have
\begin{displaymath}
x_i(k+1) = \left(A \otimes \mathbf{x}(k)\right)_{i}, \quad \mbox{for all}\:\:\:
i\in\mathcal{V}.
\end{displaymath}
With the understanding that by product we mean tropical product, the
max-consensus problem can be written as
\[ \mathbf{x}(k+1) = A \mathbf{x}(k) \]
or equivalently
\begin{equation}
\label{Consensus:Equation:MaxProtocol:InMaxPlusAlgebra}
\mathbf{x}(k) = A^k \mathbf{x}(0),
\end{equation}
 where
$\mathbf{x}(0) = (x_1(0),\ldots,x_N(0))^T\in\mathbb{R}^n$ is the
vector of initial values of nodes of $\mathcal{G}$ and by $A^k$ we
mean $\underbrace{A\otimes \cdots \otimes A}_{k \;\mathrm{ times}}$.

Equation~\eqref{Consensus:Equation:MaxProtocol:InMaxPlusAlgebra} is,
in form, similar to the linear consensus protocol \citep{olfati1,ren}
and its convergence can be analyzed by looking at the behavior of
$A^k$ for large $k$.

%===================Section: Convergence Analysis============
\section{Convergence Analysis}
\label{Consensus:SubSection:MaxConsensus:ConvergenceAnalysis}
Towards deriving the convergence criterion for max-consensus
protocol we first establish certain key properties of tropical
adjacency matrix multiplication.
\begin{definition}
{\label{Consensus:Definition:MaxConsensus:DependencyGraph} Given a
tropical adjacency matrix $A\in\mathbb{T}^{N\times N}$, its 
dependency graph, denoted as $\mathcal{G}(A)$, is a graph
$(\mathcal{V},\mathcal{E})$ with $\mathcal{V}=\{1,\ldots,N\}$ and
$(i,j)\in\mathcal{E}$ if $(A)_{ji}=0$.}
\end{definition}
In the sequel, we denote the edge set of the dependency graph
$\mathcal{G}(A)$ as $\mathcal{E}(A)$. We have the following observations.

\begin{lemma}
  {\label{Consensus:Claim:MaxConsensus:ClaimInProductofAdjecencyMatrices}
    Let $A,B\in \mathbb{T}^{N\times N}$ be tropical adjacency matrices and
    if any one of the following holds
    \begin{enumerate}
    \item $(j,i)\in\mathcal{E}(A)$,
    \item $(j,i)\in\mathcal{E}(B)$, and
    \item $\exists l$ such that $(l,i)\in\mathcal{E}(A)$ and
      $(j,l)\in\mathcal{E}(B)$
    \end{enumerate}
    then $c_{ij}=0$. }
\end{lemma}
\begin{proof}
  {\noindent Case~1:
    We have $a_{ij}=0$ as $(j,i)\in\mathcal{E}(A)$,  and $b_{jj}=0$
    since it is diagonal element. This gives us that $a_{ij} + b_{jj}=0$
    and therefore from Eq.~(\ref{sij}) (also using the fact that
    $c_{ij}\in\{0,-\infty\}$), we have
    \begin{align*}
      c_{ij} &= \max\{\max_{l\in\mathcal{V}\setminus j} (a_{il} + b_{lj}),(a_{ij} + b_{jj})\}\\
      &= \max\{\max_{l\in\mathcal{V}\setminus j} (a_{il} + b_{lj}),0\}=0.
    \end{align*}
    Case~2:
    Follows from similar arguments as in Case 1.\\
    Case~3:
    We have $a_{il}=0$ as $(l,i)\in\mathcal{E}(A)$, and $b_{lj}=0$ as
    $(j,l)\in\mathcal{E}(B)$. Therefore $a_{il} + b_{lj}=0$ which leads
    us to the conclusion that $c_{ij}=0$ using arguments similar to that
    in Case~1. }
\end{proof}

Now, we have the following lemma.
\begin{lemma}
  {\label{Consensus:Lemma:MaxConsensus:ProductofAdjecencyMatrices} Let
    $A,B\in \mathbb{T}^{N\times N}$ be tropical adjacency matrices.
    Then, $C = A\otimes B$ is also a tropical adjacency matrix.
    Moreover, $\mathcal{E}(C)\supseteq\mathcal{E}(A)\cup\mathcal{E}(B)$
    and if there exists a node $l$ such that $(l,i)\in\mathcal{E}(A)$
    and $(j,l)\in\mathcal{E}(B)$, then $(j,i)\in\mathcal{E}(C)$.}
\end{lemma}
\begin{proof}
  {Let $a_{ij},b_{ij},$ and $c_{ij}$ be the $(i,j)^\mathrm{th}$
    element of $A,B,$ and $C$ respectively. Since $C=A\otimes B$, from
    matrix multiplication rule,
    Eq.~\eqref{Consensus:Equation:MaxConsensus:TropicalMultiplication},
    we get
    \begin{align}
      c_{ij} &= \bigoplus_{l=1}^N (a_{il}\otimes b_{lj}) \label{sijtrop}\\
      &= \max_l (a_{il} + b_{lj}) \label{sij}.
    \end{align}
    Since $a_{ij},b_{ij}\in\{0,-\infty\}$ and
    $(\{0,-\infty\},\oplus,\otimes)$ is a semiring (that is, closed
    under operations $\oplus$ and $\otimes$), from Eq.~(\ref{sijtrop}) we have
    $c_{ij}\in\{0,-\infty\}$.
    If $c_{ij}=0$, then $(j,i)\in\mathcal{E}(C)$.

    %% We have the following claim.
    %%  \begin{claim}
    %% {\label{Consensus:Claim:MaxConsensus:ClaimInProductofAdjecencyMatrices}
    %%  If any one of the following holds
    %% \begin{enumerate}
    %% \item $(j,i)\in\mathcal{E}(A)$
    %% \item $(j,i)\in\mathcal{E}(B)$
    %% \item $\exists l$ such that $(l,i)\in\mathcal{E}(A)$ and
    %% $(j,l)\in\mathcal{E}(B)$
    %% \end{enumerate}
    %% then $c_{ij}=0$. }
    %% \end{claim}
    %% \begin{proof}
    %%  {\noindent Case~1:\\
    %% We have $a_{ij}=0$ as $(j,i)\in\mathcal{E}(A)$,  and $b_{jj}=0$
    %% since it is diagonal element. This gives us that $a_{ij} + b_{jj}=0$
    %% and therefore from Eq.~(\ref{sij}) (also using the fact that
    %% $c_{ij}\in\{0,-\infty\}$), we have
    %% \begin{align*}
    %%  c_{ij} &= \max\{\max_{l\in\mathcal{V}\setminus j} (a_{il} + b_{lj}),(a_{ij} + b_{jj})\}\\
    %% &= \max\{\max_{l\in\mathcal{V}\setminus j} (a_{il} + b_{lj}),0\}=0
    %%  \end{align*}
    %% Case~2:\\
    %% Follows from similar arguments as in Case 1.\\
    %% Case~3:\\
    %% We have $a_{il}=0$ as $(l,i)\in\mathcal{E}(A)$, and $b_{lj}=0$ as
    %% $(j,l)\in\mathcal{E}(B)$. Therefore $a_{il} + b_{lj}=0$ which leads
    %% us to the conclusion that $c_{ij}=0$ using arguments similar to that
    %% in Case~1. }
    %% \end{proof}
    From
    Lemma~\ref{Consensus:Claim:MaxConsensus:ClaimInProductofAdjecencyMatrices}
    and the fact that $A$ and $B$ are tropical  
    adjacency matrices, it is clear that $c_{ii} = 0$ for all
    $i\in\{1,\ldots,N\}$. Thus $C$ is a tropical adjacency matrix since
    all its entries belong to $\mathbb{T}$ and all the diagonal entries
    are zeros. Also immediate are the other claims of the lemma.}
\end{proof}

    From Lemma~\ref{Consensus:Lemma:MaxConsensus:ProductofAdjecencyMatrices},
    we observe that the tropical adjacency matrix multiplication is
    non-commutative which is inherited from the non-commutativity of
    matrix multiplication. The essence of
    Lemma~\ref{Consensus:Lemma:MaxConsensus:ProductofAdjecencyMatrices}
    is that the tropical adjacency matrix multiplication is
    superadditive with respect to the corresponding edge sets. Thus from
    this lemma, we have an immediate observation.
    \begin{proposition}
      {\label{Consensus:Proposition:MaxConsensus:ZeroEntryInvariance}The
        $0$ entries in the tropical adjacency matrices are not altered by
        adjacency matrix multiplications. Moreover, if $A$ is any adjacency
        matrix $A\mathbf{0}=\mathbf{0}A = \mathbf{0}$.}
    \end{proposition}
    In the above proposition and sequel, $\mathbf{0}$ corresponds to
    tropical adjacency matrix with all entries as zeros. An implication
    of above proposition is that, once a consensus is achieved, it
    becomes independent of the underlying graph structure. Thus, using a
    max-consensus protocol, a consensus, if achieved, is achieved in
    finite time. This is in contrast to other consensus protocols
    (average consensus protocol, for example) where the consensus is
    achieved asymptotically.
    
    We have the following definitions.
    \begin{definition}
{\label{Consensus:Definition:MaxConsensus:StronglyConnectedGraph} A
graph $\mathcal{G}=(\mathcal{V},\mathcal{E})$ is strongly
connected if for every $i,j\in\mathcal{V}$ there exists a sequence
of nodes (called a path from $i$ to $j$) in $\mathcal{V}$,
$i=i_0,i_1,\ldots,i_s=j$, such that $(i_{l-1},i_l)\in\mathcal{E}$
for $l\in\{1,\ldots,s\}$ and $s\leq N-1$.}
\end{definition}
Equivalently, in a strongly connected graph, it is possible to start
at any node and reach any other node by following the directed edges
of the graph.
\begin{definition}
{\label{Consensus:Definition:MaxConsensus:CompletelyConnectedGraph}A
completely connected graph is a graph in which there is a
directed edge from every node to every other node.}
\end{definition}
The tropical adjacency matrix that corresponds to a completely
connected graph is the zero matrix $\mathbf{0}$.
The shortest distance/path length from $i$ to $j$ is the least
number of directed edges to follow from node $i$ to reach node $j$.
\begin{definition}
{\label{Consensus:Definition:MaxConsensus:DiameterofGraph} Let
$d_{ij}$ denote the shortest path length from  node $i$ to node $j$
in a graph $\mathcal{G} = (\mathcal{V},\mathcal{E})$, where
$i,j\in\mathcal{V}$. Then, diameter of the graph, $d =
\max_{i,j\in\mathcal{V}} d_{ij}$. }
\end{definition}
If the graph is not strongly connected, there may not be a path
between some of the nodes and in that case, we assign a value of
$\infty$ to the diameter.
\begin{definition}
{\label{Consensus:Definition:MaxConsensus:kNeighbor} For all
$i,j\in\mathcal{V}$, node $j$ is a $p$-neighbor of $i$ if the
shortest path length from node $j$ to $i$ is less than or equal to
$p$. We denote the $p$-neighbor set of $i$ as $\mathcal{N}_i^p$.}
\end{definition}
From the above definition, it follows that
$\mathcal{N}_i^1=\mathcal{N}_i$ and the following recursion formula
holds
\begin{equation*}
\mathcal{N}_i^p = \left(\bigcup_{l\in\mathcal{N}_i^{p-1}}
\{j:j\in\mathcal{N}_l\} \right) \bigcup \mathcal{N}_i^{p-1}
\end{equation*}
We have the following proposition which immediately follows from the
definitions of strongly connected graph, diameter of a graph, and
the $p$-neighbor set of a node.
\begin{proposition}
{\label{Consensus:Proposition:MaxConsensus:dNeighborCompleteVertexSet}
For a strongly connected graph with vertex set $\mathcal{V}$ and
having diameter $d$, $\mathcal{N}_i^d=\mathcal{V}$ for all
$i\in\mathcal{V}$.}
\end{proposition}
For a strongly connected graph $\mathcal{G}(A)$ of diameter $d$,
$\mathcal{G}(A^d)$ is a completely connected graph as shown in the
proof of the following lemma, which leads to the main result.
\begin{lemma}
{\label{Consensus:Lemma:MaxConsensus:AdjecencyMarixNilPotent} Let
$A\in\mathbb{T}^{n\times n}$ be a tropical adjacency matrix. If
$\mathcal{G}(A)$ is strongly connected, then $A^d=\mathbf{0}$ where
$d$ is the diameter of $\mathcal{G}(A)$}
\end{lemma}
\begin{proof}
{Let $\mathcal{G}_k=\mathcal{G}(A^k)$ for $k\in\mathbb{N}$. Let
$\mathcal{N}_i^1=\mathcal{N}_i$ be the neighbor set of node $i$ in
$\mathcal{G}_1$. Using a special case of
Lemma~\ref{Consensus:Lemma:MaxConsensus:ProductofAdjecencyMatrices}
(where $B$ equal to $A$), it follows that the neighbor set of node
$i$ in $\mathcal{G}_2$ is $\mathcal{N}_i^2$. Continuing the same
argument, the neighbor set of node $i$ in graph $\mathcal{G}_k$ is
$\mathcal{N}_i^k$. For, $k=d$, by
Proposition~\ref{Consensus:Proposition:MaxConsensus:dNeighborCompleteVertexSet},
we have $\mathcal{G}_d$ as a completely connected graph and thus the
claim of lemma follows.}
\end{proof}
We have the following necessary and sufficient condition for
convergence of max-consensus protocol.
\begin{theorem}
{\label{Consensus:Theorem:MaxConsensus:ConvergenceTheorem:AdjecencyMatrix}
The max-consensus protocol in
(\ref{Consensus:Equation:MaxConsensus:MaxConsensusProtocol})
converges for all initial conditions if and only if there exists a
$k\in\mathbb{N}$ such that $A^k=\mathbf{0}$.}
\end{theorem}
\begin{proof}
{\noindent {\em if:} \\
If $A^k=\mathbf{0}$ for some $k\in\mathbb{N}$, then for all
$i\in\mathcal{V}$
\begin{align*}
x_i(k+1) &= \left(A^k \mathbf{x}(0)\right)_i \\
&= \max\{0+x_1(0),\ldots,0+x_N(0)\} \\
&= \max_{j\in\mathcal{V}}\{x_j(0)\}
\end{align*}
{\em only if:}\\
Suppose $A^k$ has at least one non-zero entry for all
$k\in\mathbb{N}$. Let this be the $(i,j)^\mathrm{th}$ entry, that
is, $\left(A\right)_{i,j} = -\infty$. So we have
\begin{align*}
x_i(k+1) &= \max\{0+x_1(0),\ldots,-\infty+x_j(0),\ldots,0+x_N(0)
\}\\
&= \max\{x_1(0),\ldots,x_{j-1}(0),x_{j+1}(0),\ldots,x_N(0)\}
\end{align*}
 This means that the node $i$ does not converge to the
maximum of all initial nodal values if $x_j(0)$ happens to be the
maximum.
  }
\end{proof}
The above result is used to prove the following theorem which is the
main result of this section.
\begin{theorem}
{\label{Consensus:Theorem:MaxConsensus:ConvergenceTheorem:GraphConnectivity}
The max-consensus protocol in
(\ref{Consensus:Equation:MaxConsensus:MaxConsensusProtocol})
converges for all initial conditions if and only if $\mathcal{G}(A)$
is strongly connected.}
\end{theorem}
\begin{proof}
{{\em if:} \\
If $\mathcal{G}(A)$ is strongly connected, from
Lemma~\ref{Consensus:Lemma:MaxConsensus:AdjecencyMarixNilPotent}, we
have that $A^d=\mathbf{0}$ and thus using
Theorem~\ref{Consensus:Theorem:MaxConsensus:ConvergenceTheorem:AdjecencyMatrix}
it follows that a consensus is achieved in $d$ time steps. \\
{\em only if:}\\
We use the fact that if $\mathcal{G}(A)$ is not strongly connected,
then $A$ is reducible \citep{fiedler,ryser}. That is, there exist a
permutation $\mathbf{\sigma}$ of the numbering of nodes in
$\mathcal{V}$ such that the resultant adjacency matrix
$\mathfrak{S}A\mathfrak{S}^T$ is of the form \[ \hat{A} = \left[
\begin{array}{cc} A_{11} & [-\mathbf{\infty}]\\ A_{12} & A_{22} \end{array} \right]\]
where $\mathfrak{S}$ is the permutation matrix \citep{ryser}
corresponding to $\sigma$, $A_{1}$ and $A_2$ are square matrices
that are non-vacuous (dimension greater than or equal to $1$), and
$[-\mathbf{\infty}]$ is a matrix of appropriate dimension with all
entries as $-\infty$.
\begin{claim}
{\label{Consensus:Claim:MaxConsensus:NonExistenceofNilPotent} There
exist no $k\in\mathbb{N}$ such that, for a matrix $\hat{A}$ of above
form, $\hat{A}^k=\mathbf{0}$}
\end{claim}
\begin{proof}
{This follows from the tropical matrix multiplication of irreducible
matrices \citep{ryser}. In fact, such matrices retains above form
after multiplication.}
\end{proof}
Thus if $\mathcal{G}(A)$ is not strongly connected, then there does
not exist a $k\in\mathbb{N}$ such that $A^k=\mathbf{0}$ and thus a
max-consensus is not achieved.}
\end{proof}
%====================================================
\section{Switching Topology}
\label{Consensus:SubSection:MaxConsensus:SwitchingTopology} Now, we
look at the case of switching topology where the underlying graph
changes during each time step. We consider a consensus protocol of
the form
\begin{equation}
\mathbf{x}(k+1) = A_k \mathbf{x}(k)
\label{Consensus:Equation:MaxConsensus:SwitchingTopology}
\end{equation}
where $A_k \in\mathbb{T}^{N\times N}$ is a tropical adjacency matrix
for $k\in\mathbb{Z}_{\ge 0}$. Equivalently, we have
\[ \mathbf{x}(k+1) = A_kA_{k-1}\cdots A_0 \mathbf{x}(0) \]
\begin{theorem}
{\label{Consensus:Theorem:MaxConsensus:SwitchingTopologyConvergenceTheorem:MatrixProducts}
The consensus protocol in
(\ref{Consensus:Equation:MaxConsensus:SwitchingTopology}) converges
for all initial conditions if and only if there exist a
$k\in\mathbb{N}$ such that $A_kA_{k-1}\cdots A_0 = \mathbf{0}$.}
\end{theorem}
\begin{proof}
{Proof is similar to that of
Theorem~\ref{Consensus:Theorem:MaxConsensus:ConvergenceTheorem:AdjecencyMatrix}
where $A^k$ is replaced by $A_kA_{k-1}\cdots A_0$.}
\end{proof}
If $A_k$ is drawn from a finite set $\{A_1,A_2,\ldots,A_m\}$ of
tropical adjacency matrices, an interesting question to ask is
\begin{quote}
Does there exist a finite sequence $i_1,i_2,\ldots,i_n$ with $1\leq
i_r\leq m$ and $r\in\{1,\ldots,n\}$ such that the system in
Eq.~(\ref{Consensus:Equation:MaxConsensus:SwitchingTopology})
converges in $n$ time steps?
\end{quote}
Before answering that we give the following useful definition.
\begin{definition}
{\label{Consensus:Definition:MaxConsensus:JointlyStronglyConnected}
The graphs
$\mathcal{G}_1(\mathcal{V},\mathcal{E}_1),\ldots,\mathcal{G}_m(\mathcal{V},\mathcal{E}_m)$
with adjacency matrices $A_1,\ldots,A_m$ are called jointly
strongly connected if the union graph
$\mathcal{G}(A_1\oplus\cdots\oplus A_m)=\bigcup_{r=1}^m
\mathcal{G}_r$, with vertex and edge set as
$(\mathcal{V},\bigcup_{r=1}^m\mathcal{E}_r)$, is strongly
connected.}
\end{definition}
\begin{proposition}
{\label{Consensus:Proposition:MaxConsensus:StronglyImpliesJointly}
Let $A,B\in\mathbb{T}^{N\times N}$ be tropical adjacency matrices.
Then the graph $\mathcal{G}(A\otimes B)$ is strongly connected if
and only if $\mathcal{G}(A)$ and $\mathcal{G}(B)$ are jointly
strongly connected.}
\end{proposition}
\begin{proof}
{\noindent {\em if:}\\Follows from the superadditivity property of
the tropical adjacency matrix multiplication proved in Lemma~\ref{Consensus:Lemma:MaxConsensus:ProductofAdjecencyMatrices}. \\
{\em only if:}\\ Follows from
Claim~\ref{Consensus:Claim:MaxConsensus:ClaimInProductofAdjecencyMatrices}
in the proof of
Lemma~\ref{Consensus:Lemma:MaxConsensus:ProductofAdjecencyMatrices}.
The claim is actually `if and only if' although the `only if' part
is not required for
Lemma~\ref{Consensus:Lemma:MaxConsensus:ProductofAdjecencyMatrices}.
In fact, the three ways in which $c_{ij}$ can become zero (refer to
Claim~\ref{Consensus:Claim:MaxConsensus:ClaimInProductofAdjecencyMatrices})
are the only ways in which it will be zero. As far as strong
connectivity is concerned, condition 3 in
Claim~\ref{Consensus:Claim:MaxConsensus:ClaimInProductofAdjecencyMatrices}
is immaterial because reachability of a node is only what matters
for strong connectivity in which case the third condition is just
redundant. Thus, the strong connectivity of $\mathcal{G}(A\otimes
B)$ is same as that of $\mathcal{G}(A)\bigcup\mathcal{G}(B)$.}
\end{proof}
Thus, if the graphs corresponding to the tropical adjacency matrices
$\{A_1,\ldots,A_m\}$ are jointly strongly connected, then
$(A_1\otimes\cdots\otimes A_m)^n = \mathbf{0}$ for some
$n\in\mathbb{N}$. Following theorem is the answer to the question of
convergence of max-consensus protocol under switching topology
(Eq.~\eqref{Consensus:Equation:MaxConsensus:SwitchingTopology}).
\begin{theorem}
{\label{Consensus:Theorem:MaxConsensus:SwitchingTopologyNecessary&SufficientCondition}
Given a finite set of tropical adjacency matrices
$\{A_1,\ldots,A_m\}\subset\mathbb{T}^{N\times N}$, there exists a
finite sequence $A_{i_1},\ldots,A_{i_n}$ with $1\leq i_r\leq m$ for
all $r=\{1,\ldots,n\}$ such that the system in
Eq.~(\ref{Consensus:Equation:MaxConsensus:SwitchingTopology}) with
$A_k\in\{A_1,\ldots,A_m\}$ attains consensus for all initial
conditions if and only if $\mathcal{G}(A_1),\ldots,\mathcal{G}(A_m)$
are jointly strongly connected.}
\end{theorem}
\begin{proof}
{Follows by extending the result of
Proposition~\ref{Consensus:Proposition:MaxConsensus:StronglyImpliesJointly}
to finite case and proceeding in the similar lines of proof of
Theorem~\ref{Consensus:Theorem:MaxConsensus:ConvergenceTheorem:GraphConnectivity}.}
\end{proof}
The above theorem has a couple of interesting corollaries. We give
the following definition towards this.
\begin{definition}
{\label{Consensus:Definition:MaxConsensus:MatrixMortalityProblem}
Given a finite set of matrices $\{A_1,\ldots,A_m\}$, the matrix
mortality problem asks the following question: Does there exist a
finite sequence $i_1,\ldots,i_n$ with $1\leq i_r\leq m$ and
$r\in\{1,\ldots,n\}$ such that $A_{i_1},\ldots,A_{i_r}=\mathbf{0}$?}
\end{definition}
The matrix mortality problem is of great interest to the theoretical
computer science community \citep{sipser}. The matrix mortality
problem is usually undecidable, that is, there does not exist an
algorithm which can find such a sequence given a matrix set as
input.
\begin{corollary}
{\label{Consensus:Corollary:MaxConsensus:DecidabilityofMatrixMortality}
If the finite set of input matrices $\{A_1,\ldots,A_m\}$ are
tropical adjacency matrices, then the matrix mortality problem is
decidable. Moreover, this can be done in polynomial time.}
\end{corollary}
From
Theorem~\ref{Consensus:Theorem:MaxConsensus:SwitchingTopologyNecessary&SufficientCondition},
we know that the matrix mortality problem has a positive answer if
and only if $A_1,\ldots,A_m$ are jointly strongly connected. In that
case, the zero matrix is obtained as $(A_1\otimes\cdots\otimes
A_m)^n$ for some finite $n\in\mathbb{N}$. Another interesting
corollary is as follows.
\begin{corollary}
{\label{Consensus:Corollary:MaxConsensus:SemigroupMonoid} The
semigroup (under binary operation $\otimes$ which is associative)
generated by the tropical adjacency matrices $\{A_1,$ $\ldots,$
$A_m\}\subset\mathbb{T}^{N\times N}$ contains $\mathbf{0}$ if and
only if $\bigcup_{r=1}^m \mathcal{G}(A_r)$ is strongly connected.}
\end{corollary}

The following remarks are in place.

A much more stronger result on the convergence of the max-consensus
protocol can be achieved if the condition in our theorems `achieves
consensus for all initial condition' is relaxed. In particular, a
system $\mathbf{x}(k+1) = A_k \mathbf{x}(0)$ achieves consensus for
a particular input $\mathbf{x}(0)$ at time step $k+1$ if and only if
there exists a vector $\mathbf{y} \in \mathbb{T}^N$ such that
$A_{k}\mathbf{y} = \mathbf{0}$, where
\[ y_{i} = \left\{\begin{array}{ll}
0 & \text{if } x_i(0) = \max_j x_j(0) \\
-\infty & \text{otherwise}
\end{array} \right. \]
In other words, if node $i$ has a maximum initial value, the
necessary and sufficient condition for a max-consensus to occur at
the $(k+1)^{\rm th}$ time step is that all elements in the $i^{\rm
th}$ column of $A^k$ are zeros or equivalently, the graph
$\mathcal{G}(A^k)$ has a directed spanning tree rooted at node $i$.

In the case of the max-consensus protocol, if a consensus occurs,
it will happen in a finite number of steps. This should be
contrasted with the asymptotic convergence of weighted average
consensus protocols \citep{olfati1,ren}. This property enables us to
give a stronger convergence criterion. In fact, we gave the
necessary and sufficient condition for convergence under switching
topology. However, this is difficult in general for the weighted
average consensus case as one has to consider infinite product of
matrices (as opposed to product of finite number of matrices in the
max-consensus case) while analyzing convergence. Thus, in case of
weighted average consensus, one is forced to impose stronger
requirements on the matrices to achieve convergence of the infinite
matrix products. This will, in general, result in obtaining only a
sufficiency condition for convergence.

A max-consensus, scheduled at regular intervals, over large
networks can be used to detect network faults. This is practical as
most of the naturally occurring networks like random networks
\citep{erdos}, small world networks \citep{watts}, and scale free
networks \citep{barabasi} usually have very small diameters
\citep{princeton}. Since the number of time steps for max-consensus
algorithm to converge is equal to the diameter of the underlying
network, consensus is achieved fast in these networks due to the
`small diameter property'. A failure to attain consensus in the
specified time is a indication to the presence of a network fault.

The matrix operations in the semiring
$(\mathbb{T},\max,+)$ is exactly same as the matrix operations in
the boolean semiring $(\{0,1\},\mathrm{OR},\mathrm{AND})$
\citep{hammer} as there is a ring-isomorphism between these two
semirings ($(\mathbb{T},\max,+)$ is isomorphic to
$(\{0,1\},\mathrm{OR},\mathrm{AND})$). Thus, although the
max-consensus problem cannot be posed using boolean semiring, all
the matrix multiplication properties used to prove convergence can
be identically obtained by working in boolean semiring (See
\citep{hammer} for properties of boolean adjacency matrix
multiplication).

%========================================
\section{Concluding Remarks}
We analyzed the convergence of max consensus protocol with both
fixed and switching topologies. The observation that tropical
algebra gives a natural way to formulate this problem enabled an
analysis of the convergence in terms of tropical matrix products
which could be easily extended to switching topology case. We gave
the necessary and sufficient condition for max-consensus to occur in
fixed as well as switched topology networks.

{\footnotesize
\bibliographystyle{jtbnew}
\bibliography{joe,papi}
}
\end{document}